\documentclass[12pt,a4paper]{article}
\usepackage[latin1]{inputenc}
\usepackage{amsmath}
\usepackage{amsfonts}
\usepackage{amssymb}
\usepackage{amsthm}

\newcommand{\congruent}[4]{#1 #2 \equiv #3 #4}
\newcommand{\betweenness}[3]{\mathrm{B}#1 #2 #3}

\newcommand{\universal}[2]{\forall #1.\; #2}
\newcommand{\existential}[2]{\exists #1.\; #2}

\newcommand{\axiomset}[1]{\textnormal{\textup{\textsf{#1}}}}
\newcommand{\mathaxiomset}[1]{\mathnormal{\mathsf{#1}}}
\newcommand{\cetwo}{\axiomset{CE$_\mathaxiomset{2}$}}
\newcommand{\cetwoprime}{\axiomset{CE$\mathaxiomset{_2'}$}}
\newcommand{\atasol}{\axiomset{A}}
\newcommand{\aprime}{\axiomset{A$\mathaxiomset{'}$}}
\newcommand{\ax}{\textup{(Ax)}}
\newcommand{\model}[1]{$\mathcal{#1}$}

\newtheorem{theorem}{Theorem}
\newtheorem{lemma}[theorem]{Lemma}
\newtheorem{corollary}[theorem]{Corollary}
\newcommand{\lemlabel}[1]{\label{lem:#1}}
\newcommand{\lemref}[1]{Lemma \ref{lem:#1}}
\newcommand{\thmlabel}[1]{\label{thm:#1}}
\newcommand{\thmref}[1]{Theorem \ref{thm:#1}}
\newcommand{\corlabel}[1]{\label{cor:#1}}

\author{T.~J.~M.~Makarios\footnote{\href{mailto:tjm1983@gmail.com}{\nolinkurl{tjm1983@gmail.com}}}}
\title{A further simplification of Tarski's axioms of geometry}

\usepackage[colorlinks=true, linkcolor=black, citecolor=black, urlcolor=black]{hyperref}

\begin{document}

\maketitle

\begin{abstract}
A slight modification to one of Tarski's axioms of plane Euclidean geometry is proposed.
This modification allows another of the axioms to be omitted from the set of axioms and proven as a theorem.
This change to the system of axioms simplifies the system as a whole, without sacrificing the useful modularity of some of its axioms.
The new system is shown to possess all of the known independence properties of the system on which it was based;
in addition, another of the axioms is shown to be independent in the new system.
\end{abstract}

\section{Background}
Alfred Tarski's axioms of geometry were first described in a course he gave at the University of Warsaw in 1926--1927.
Since then, they have undergone numerous improvements, with some axioms modified, and other superfluous axioms removed;
for a history of the changes, see \cite{givant} (especially Section 2), or for a summary, see Figure 2 in \cite{narboux}.

The axioms rely on only one primitive notion --- that of points --- and two primitive relations: betweenness and congruence.
Congruence is denoted $\congruent{a}{b}{c}{d}$, and can be interpreted as asserting that the line segment from $a$ to $b$ is congruent to the line segment from $c$ to $d$.
Betweenness is denoted $\betweenness{a}{b}{c}$, and can be interpreted as asserting that $b$ lies on the segment from $a$ to $c$ (and may be equal to $a$ or $c$).

The version of the axioms used in \cite{schwabhauser} (see pages 10--14) consists of ten first-order axioms, together with either a first-order axiom schema, or a single higher-order axiom.
This version has been adopted in later publications, such as \cite{mechindtarski} (see Sections 2.3 and 2.4) and \cite{narboux} (see Figure 3), and Victor Pambuccian has called it the ``most polished form'' of Tarski's axioms (see \cite{pambuccianabsolute}, page 122).

This semi-canonical version of the system is the result of many simplifications to the original system of twenty axioms plus one axiom schema.
At least one of these simplifications appears to have taken the form of a slight alteration to one axiom in order to allow another axiom to be dropped and subsequently proven as a theorem.

Specifically, in \cite{decision} (see note 18), axiom (ix) is a version of what is called the \emph{axiom of Pasch}, which states (modulo notational differences):
\[ \universal{a, b, c, p, q}{\betweenness{a}{p}{c} \wedge \betweenness{q}{c}{b} \longrightarrow \existential{x}{\betweenness{a}{x}{q} \wedge \betweenness{x}{p}{b}}} \tag{OP}\label{op} \]
In \cite{whatis}, this has been replaced by axiom A9, which states (again, modulo notational changes):
\[ \universal{a, b, c, p, q}{\existential{x}{\betweenness{a}{p}{c} \wedge \betweenness{q}{c}{b} \longrightarrow \betweenness{a}{x}{q} \wedge \betweenness{b}{p}{x}}} \tag{OP$'$}\label{opp} \]
The significant change between the two is the reversal of the order of points in the final betweenness relation.
They are easily shown to be equivalent using the symmetry of betweenness, which states:
\[ \universal{a, b, c}{\betweenness{a}{b}{c} \longrightarrow \betweenness{c}{b}{a}} \tag{SB}\label{sb} \]

The interesting thing is that in \cite{decision}, \eqref{sb} is an axiom (axiom (iii)), but in \cite{whatis}, it has been removed from the list of axioms.
Haragauri Gupta's proof of \eqref{sb} (see \cite{gupta}, Theorem 2.18) relies on the precise ordering of points in the final betweenness relation in \eqref{opp}.

Also, Wolfram Schwabh\"auser, Wanda Szmielew, and Alfred Tarski, on page 12 of \cite{schwabhauser}, draw attention to the ordering of points in their version of the axiom of Pasch (labelled \eqref{pa} in this paper), noting that it is important until after the proof of \eqref{sb} (which they call Satz 3.2).
Again, their proof of \eqref{sb} (which is essentially the same as this paper's proof of \lemref{sb}) relies on the precise ordering of points in \eqref{pa}.

Therefore, although it does not appear to be explicitly acknowledged in the published literature, it seems likely that the change from \eqref{op} to \eqref{opp} was necessary to allow the removal of \eqref{sb} from the set of axioms.
It may be the case that \eqref{sb} is a theorem even with \eqref{op}, and that this was not known when it was replaced by \eqref{opp}, or that it was known, but \eqref{opp} allows a simpler proof.

In any case, it appears that Tarski was willing to reorder points in his axioms to allow the simplification of the axiom system as a whole, either by removing axioms or merely by simplifying proofs of theorems.

In the tradition of such simplifications, this paper presents one further simplification of the axiom system;
one of the axioms is slightly modified, allowing another of the traditional axioms to be proven as a theorem, rather than assumed as an axiom.

\section{The axioms}

Tarski's axioms, as stated in \cite{schwabhauser}, pages 10--14, are as follows.
The names are adopted from \cite{mechindtarski} (Section 2.4), which provides some diagrams and intuitive explanations for the axioms.
\begin{itemize}
\item Reflexivity axiom for equidistance
\[ \universal{a, b}{\congruent{a}{b}{b}{a}} \tag{RE}\label{re} \]
\item Transitivity axiom for equidistance
\[ \universal{a, b, p, q, r, s}{\congruent{a}{b}{p}{q} \wedge \congruent{a}{b}{r}{s} \longrightarrow \congruent{p}{q}{r}{s}} \tag{TE}\label{te} \]
\item Identity axiom for equidistance
\[ \universal{a, b, c}{\congruent{a}{b}{c}{c} \longrightarrow a = b} \tag{IE}\label{ie} \]
\item Axiom of segment construction
\[ \universal{a, b, c, q}{\existential{x}{\betweenness{q}{a}{x} \wedge \congruent{a}{x}{b}{c}}} \tag{SC}\label{sc} \]
\item Five-segments axiom
\begin{align*}
\universal{a, b, c, d, a', b', c', d'}&{a \neq b \wedge \betweenness{a}{b}{c} \wedge \betweenness{a'}{b'}{c'}} \\
& \qquad \wedge \congruent{a}{b}{a'}{b'} \wedge \congruent{b}{c}{b'}{c'} \wedge \congruent{a}{d}{a'}{d'} \wedge \congruent{b}{d}{b'}{d'} \\
& \longrightarrow \congruent{c}{d}{c'}{d'} \tag{FS}\label{fs}
\end{align*}
\item Identity axiom for betweenness
\[ \universal{a, b}{\betweenness{a}{b}{a} \longrightarrow a = b} \tag{IB}\label{ib} \]
\item Axiom of Pasch
\[ \universal{a, b, c, p, q}{\betweenness{a}{p}{c} \wedge \betweenness{b}{q}{c} \longrightarrow \existential{x}{\betweenness{p}{x}{b} \wedge \betweenness{q}{x}{a}}} \tag{Pa}\label{pa} \]
\item Lower $2$-dimensional axiom
\[ \existential{a, b, c}{\neg \betweenness{a}{b}{c} \wedge \neg \betweenness{b}{c}{a} \wedge \neg \betweenness{c}{a}{b}} \tag{Lo$_2$}\label{l2} \]
\item Upper $2$-dimensional axiom
\[ \universal{a, b, c, p, q}{p \neq q \wedge \congruent apaq \wedge \congruent bpbq \wedge \congruent cpcq \longrightarrow \betweenness abc \vee \betweenness bca \vee \betweenness cab} \tag{Up$_2$}\label{u2} \]
\item Euclidean axiom
\[ \universal{a, b, c, d, t}{\betweenness adt \wedge \betweenness bdc \wedge a \neq d \longrightarrow \existential{x, y}{\betweenness abx \wedge \betweenness acy \wedge \betweenness xty}} \tag{Eu}\label{eu} \]
\item Axiom of continuity
\begin{align*}
\universal{X, Y}&{(\existential{a}{\universal{x, y}{x \in X \wedge y \in Y \longrightarrow \betweenness axy}})} \\
 & \longrightarrow (\existential{b}{\universal{x, y}{x \in X \wedge y \in Y \longrightarrow \betweenness xby}}) \tag{Co}\label{co}
\end{align*}
\end{itemize}

This collection of eleven axioms, Tarski's axioms of the continuous Euclidean plane, will be denoted \cetwo.
A similar collection of axioms, denoted \cetwoprime, is obtained by removing \eqref{re} from \cetwo{} and replacing \eqref{fs} with
\begin{align*}
\universal{a, b, c, d, a', b', c', d'}&{a \neq b \wedge \betweenness{a}{b}{c} \wedge \betweenness{a'}{b'}{c'}} \\
& \qquad \wedge \congruent{a}{b}{a'}{b'} \wedge \congruent{b}{c}{b'}{c'} \wedge \congruent{a}{d}{a'}{d'} \wedge \congruent{b}{d}{b'}{d'} \\
& \longrightarrow \congruent dc{c'}{d'} \tag{FS$'$}\label{fsp}
\end{align*}
The only difference between \eqref{fs} and \eqref{fsp} is the reversal of the first two points in the last congruence relation.

This paper will show that \cetwoprime{} is equivalent to \cetwo, and will, in fact, show a stronger result --- \thmref{aaprime} --- about a smaller set of axioms.

One of the features of Tarski's axiom system is its modularity:
some texts omit or delay the introduction of \eqref{co} (see, for example, \cite{archimedean}, page 61);
\eqref{eu} can be replaced by another axiom in order to investigate hyperbolic geometry, or omitted entirely, for absolute geometry (see \cite{pambuccian}, pages 331--333);
\eqref{l2} and \eqref{u2} can be replaced by other axioms that characterize other dimensions (see \cite{whatis}, footnote 5).
For this reason, let \atasol{} denote the collection of axioms \eqref{re}, \eqref{te}, \eqref{ie}, \eqref{sc}, \eqref{fs}, \eqref{ib}, and \eqref{pa}.
These are Tarski's axioms of absolute dimension-free geometry without the axiom of continuity.
Let \aprime{} denote the collection of axioms \eqref{te}, \eqref{ie}, \eqref{sc}, \eqref{fsp}, \eqref{ib}, and \eqref{pa}.

The stronger result that this paper shows is that \aprime{} is equivalent to \atasol.
Thus, the modularity of axioms \eqref{l2}, \eqref{u2}, \eqref{eu}, and \eqref{co} is unaffected by the proposed change to Tarski's axiom system.

\section{Proof of equivalence}

\begin{lemma}\lemlabel{abab}
If \eqref{te} and \eqref{sc} hold, then given any points $a$ and $b$, we have $\congruent abab$.
\end{lemma}
\begin{proof}
Given $a$ and $b$, \eqref{sc} lets us obtain a point $x$ such that $\congruent axab$.
Using this twice in \eqref{te} gives us $\congruent abab$.
\end{proof}

\begin{lemma}\lemlabel{cdab}
If \eqref{te} and \eqref{sc} hold and $a$, $b$, $c$, and $d$ are points such that $\congruent abcd$, then $\congruent cdab$.
\end{lemma}
\begin{proof}
By \lemref{abab}, we have $\congruent abab$.
Using $\congruent abcd$ and $\congruent abab$, \eqref{te} tells us that $\congruent cdab$.
\end{proof}

\begin{lemma}\lemlabel{abb}
If \eqref{ie} and \eqref{sc} hold, then given any points $a$ and $b$, we have $\betweenness abb$.
\end{lemma}
\begin{proof}
Given $a$ and $b$, \eqref{sc} lets us obtain a point $x$ such that $\betweenness abx$ and $\congruent bxbb$.
Then \eqref{ie} tells us that $b = x$, so $\betweenness abb$.
\end{proof}

\begin{lemma}\lemlabel{sb}
\eqref{ie}, \eqref{sc}, \eqref{ib}, and \eqref{pa} together imply \eqref{sb}.
\end{lemma}
\begin{proof}
Suppose we are given points $a$, $b$, and $c$ such that $\betweenness abc$.
We also have $\betweenness bcc$, by \lemref{abb}.
Then \eqref{pa} lets us obtain a point $x$ such that $\betweenness bxb$ and $\betweenness cxa$.
According to \eqref{ib}, the former implies $b = x$, so the latter tells us that $\betweenness cba$.
\end{proof}

\begin{lemma}\lemlabel{re}
\aprime{} implies \eqref{re}.
\end{lemma}
\begin{proof}
Given arbitrary points $a$ and $b$, \eqref{sc} lets us obtain a point $x$ such that $\betweenness bax$ and $\congruent axba$.
We consider two cases: $x = a$ and $x \neq a$.

If $x = a$, then $\congruent aaba$.
By \lemref{cdab}, we have $\congruent baaa$, so by \eqref{ie}, we have $b = a$.
Substituting this back into the congruence as necessary gives us $\congruent abba$, as desired.

Suppose, on the other hand, that $x \neq a$.
\lemref{sb} and $\betweenness bax$ tell us that $\betweenness xab$.
\lemref{abab} tells us that $\congruent xaxa$, $\congruent abab$, and $\congruent aaaa$.
We make the following substitutions in \eqref{fsp}: $a, a' \mapsto x$; $b, b', d, d' \mapsto a$; and $c, c' \mapsto b$.
Then all of the hypotheses of \eqref{fsp} are satisfied, and its conclusion is that $\congruent abba$.
\end{proof}

\begin{lemma}\lemlabel{fsfsp}
If \eqref{re} and \eqref{te} hold, then \eqref{fsp} is equivalent to \eqref{fs}.
\end{lemma}
\begin{proof}
Because the hypotheses of \eqref{fs} and \eqref{fsp} are identical, we need only show that their conclusions are equivalent.

\eqref{re} tells us that $\congruent cddc$ and $\congruent dccd$.

If $\congruent cd{c'}{d'}$, then $\congruent cddc$ together with this fact and \eqref{te} let us conclude that $\congruent dc{c'}{d'}$.

Similarly, if $\congruent dc{c'}{d'}$, then $\congruent dccd$ together with this fact and $\eqref{te}$ let us conclude that $\congruent cd{c'}{d'}$.

Therefore \eqref{fsp} is equivalent to \eqref{fs}.
\end{proof}

\begin{theorem}\thmlabel{aaprime}
\aprime{} is equivalent to \atasol.
\end{theorem}
\begin{proof}
By Lemmas \ref{lem:re} and \ref{lem:fsfsp}, \aprime{} implies \eqref{re} and \eqref{fs}.
\aprime{} contains all of the other axioms of \atasol, so \aprime{} implies \atasol.

By \lemref{fsfsp}, \atasol{} implies \eqref{fsp}, and it contains all of the other axioms of \aprime, so \atasol{} implies \aprime.

Therefore \aprime{} is equivalent to \atasol.
\end{proof}

As an immediate corollary, we have the following:
\begin{corollary}
\cetwoprime{} is equivalent to \cetwo.\hfill\qed
\end{corollary}

\section{Independence results}

The first part of Section 5 of \cite{givant} (see pages 199 and 200) concerns the independence of Tarski's axioms.
One problem seen there is that the various historical changes to Tarski's axioms often force a reconsideration of previously established independence results.
This paper's suggested simplification of Tarski's axioms is no exception, so this section aims to establish which of the known independence results apply to the specific set of axioms \cetwoprime.

Because \cetwo{} and \cetwoprime{} differ only in their subsets \atasol{} and \aprime, \thmref{aaprime} tells us that the axioms in \cetwoprime{} but not in \aprime{} are independent if and only if they are independent in \cetwo.
In fact, we can go further than this.

\begin{theorem}\thmlabel{independencetransfer}
Suppose that \ax{} is an axiom of \cetwoprime{} other than \eqref{te} or \eqref{fsp}.
If \ax{} is independent in \cetwo{} then \ax{} is also independent in \cetwoprime.
\end{theorem}
\begin{proof}
Note that \ax{} is not \eqref{te}, because this was explicitly excluded;
nor is it \eqref{re} or \eqref{fs}, because these are not axioms of \cetwoprime, from which \ax{} was chosen.
Therefore $\cetwo\setminus\left\lbrace\ax\right\rbrace$ contains \eqref{re}, \eqref{te}, and \eqref{fs}, so by \lemref{fsfsp}, $\cetwo\setminus\left\lbrace\ax\right\rbrace \vdash \eqref{fsp}$.

All of the axioms of \cetwoprime{} other than \eqref{fsp} are also axioms of \cetwo, so $\cetwo\setminus\left\lbrace\ax\right\rbrace \vdash \cetwoprime\setminus\left\lbrace\ax\right\rbrace$.
Therefore, if $\cetwoprime\setminus\left\lbrace\ax\right\rbrace \vdash \ax$, then $\cetwo\setminus\left\lbrace\ax\right\rbrace \vdash \ax$.
Taking the contrapositive of this statement, we see that if \ax{} is independent in \cetwo{} then it is independent in \cetwoprime.
\end{proof}

This allows us to adopt almost all of the independence results known for \cetwo.

\begin{corollary}\corlabel{importedindependence}
\eqref{sc}, \eqref{ib}, \eqref{l2}, \eqref{u2}, \eqref{eu}, and \eqref{co} are each individually independent in \cetwoprime.
\end{corollary}
\begin{proof}
Schwabh\"auser and his colleagues note the independence of each of these axioms in \cetwo;
see \cite{schwabhauser}, page 26.
\end{proof}

The remaining independence result noted in \cite{schwabhauser} can also be adapted to \cetwoprime:

\begin{theorem}\thmlabel{fspindependent}
\eqref{fsp} is independent in \cetwoprime.
\end{theorem}
\begin{proof}
Schwabh\"auser and his co-authors note the existence of a model \model M demonstrating the independence of \eqref{fs} in \cetwo;
see \cite{schwabhauser}, page 26.
Because \model M satisfies \eqref{re} and \eqref{te}, but violates \eqref{fs}, we can conclude, using \lemref{fsfsp}, that \model M also violates \eqref{fsp}.
\model M satisfies every other axiom of \cetwoprime{}, because all such axioms are also axioms of \cetwo{} (and are not equal to \eqref{fs});
therefore, \model M demonstrates the independence of \eqref{fsp} in \cetwoprime.
\end{proof}

Because of the absence of \eqref{re} from \cetwoprime, we can easily show the independence of \eqref{te} in that axiom system:

\begin{theorem}\thmlabel{teindependent}
\eqref{te} is independent in \cetwoprime.
\end{theorem}
\begin{proof}
We proceed as is usual in independence proofs, by defining a model \model M of every axiom of \cetwoprime{} except \eqref{te}.
As in the real Cartesian plane (the standard model of \cetwo, and hence of \cetwoprime), we take $\mathbb{R}^2$ to be the set of points, and define betweenness so that $\betweenness abc$ if and only if $b$ is a convex combination of $a$ and $c$.
Departing from the standard model, we define congruence so that $\congruent abcd$ if and only if $a = b$.

Because the real Cartesian plane is a model of \cetwoprime, and \model M differs from the real Cartesian plane only in its definition of congruence, we can conclude that \model M is a model of all of the axioms of \cetwoprime{} that make no mention of congruence.
Those axioms are \eqref{ib}, \eqref{pa}, \eqref{l2}, \eqref{eu}, and \eqref{co}.

The definition of congruence ensures that \model M trivially satisfies \eqref{ie}.

Choosing $x = a$ ensures that \eqref{sc} is satisfied.

The hypotheses of \eqref{fsp} include $a \neq b$ and $\congruent ab{a'}{b'}$, which implies $a = b$;
this contradiction in the hypotheses means that \eqref{fsp} is vacuously true.

The hypotheses of \eqref{u2} imply that $a = b = c = p$;
it is always the case that $\betweenness ppp$, so \eqref{u2} is satisfied.

Finally, in \model M, it is the case that $\congruent{(0,0)}{(0,0)}{(0,0)}{(0,1)}$, but not the case that $\congruent{(0,0)}{(0,1)}{(0,0)}{(0,1)}$, so \model M does not satisfy \eqref{te}.

Because \model M satisfies every axiom of \cetwoprime{} except \eqref{te}, we can conclude that \eqref{te} is independent in \cetwoprime.
\end{proof}

Notice that \model M in the proof of \thmref{teindependent} also violates \eqref{re} (because it is not the case that $\congruent{(0,0)}{(0,1)}{(0,1)}{(0,0)}$), so this model would not demonstrate the independence of \eqref{te} in \cetwo, which is, as far as the author is aware, still an open question.

We now have independence results for all of the axioms of \cetwoprime{} except \eqref{ie} and \eqref{pa}.
As far as the author knows, the independence of these in \cetwo{} is still an open question (see also \cite{givant}, pages 199 and 200), although there are independence results relating to these axioms in other versions of Tarski's axiom system.

Gupta shows the independence of \eqref{ie} (his axiom A5; see \cite{gupta}, pages 41 and 41a), but only in the context of a particular variant of Tarski's axiom system.
This variant uses a weaker but more complicated form of the upper $2$-dimensional axiom, Gupta's A$'$12;\footnote{It seems that A$'$12 in Gupta's thesis ought to include $u \neq v$ among the hypotheses, as A12 does.
Taken exactly as it is printed, A$'$12 is violated by the real Cartesian plane whenever $x$, $y$, and $z$ are any non-collinear points and $u = v$.}
his independence model for \eqref{ie} also violates \eqref{u2}, which is trivially equivalent to his original axiom A12.

Les\l{}aw Szczerba established the independence of a version of the axiom of Pasch within a certain variant of Tarski's axiom system (see \cite{szczerba}).
This variant used, instead of \eqref{eu}, an axiom essentially asserting that any three non-collinear points have a circumcentre.\footnote{There appears to be a typographical error in the statement of this axiom in \cite{szczerba} (page 492).
His axiom A8$'$ states (in our notation):
\[ \universal{a, b, c}{\existential{p}{\neg(\betweenness abc \vee \betweenness bca \vee \betweenness cab) \longrightarrow \congruent apbp \wedge \congruent bpcb}} \]
It seems that the second congruence relation in the consequent should state $\congruent bpcp$;
see also \cite{givant}, pages 199 and 184.}

\section{Conclusions}

Although this paper has not answered the question of whether \eqref{re} is independent in \cetwo, it has demonstrated that \eqref{re} is superfluous to Tarski's axioms of geometry.
A slight modification to \eqref{fs} allows \eqref{re} to be proven as a theorem, and therefore removed from the set of axioms.
This simplification of the axiom system (to \cetwoprime) does not diminish its deductive power or the important ways in which it exhibits modularity.

Besides removing one of the axioms not known to be independent in \cetwo, \cetwoprime{} allows an easy proof of the independence of \eqref{te}, which was also not known to be independent in \cetwo.
The two remaining independence questions for \cetwoprime{} are, as far as the author knows, still open questions for \cetwo;
furthermore, if either axiom is shown to be independent in \cetwo, then \thmref{independencetransfer} would immediately establish its independence in \cetwoprime.

As well as trying to resolve these remaining independence questions, future work might seek other slight modifications to the axioms that may allow even known independent axioms to be dropped.
That this may be possible can be seen by considering Gupta's fully independent version of Tarski's axioms for plane Euclidean geometry (see \cite{gupta}, pages 41--41c).

His version had eleven axioms,\footnote{Strictly speaking, as he presented it, it had infinitely many axioms, but if his axiom schema A$'$13 is replaced by a comparable second-order axiom, then there are eleven axioms in total, and the second-order axiom is shown to be independent by his example on page 41c.}
some of which were deliberately made more complicated in order to allow easy proofs of the independence of other axioms (see, for example, his note on the independence of his axiom A7 on page 41b).
\cetwo{} already has \textit{simpler} axioms than Gupta's system (which he used only for the demonstration that a fully independent system is possible), but \cetwoprime{} now shows that a reduction in the \textit{number} of axioms is also possible, without making any of them more complex, despite Gupta's system being fully independent.

\bibliographystyle{alpha}
\bibliography{../bib}

\end{document}